\documentclass[11pt,reqno]{amsart}
\usepackage[a4paper, hmargin={2.7cm,2.7cm},vmargin={3.3cm,3.3cm}]{geometry}
\usepackage[english]{babel}
\usepackage{color}
\usepackage[noadjust]{cite}
\usepackage{amsmath}
\usepackage{amssymb}
\usepackage{amsfonts}
\usepackage{amsthm}
\usepackage{enumerate}
\usepackage{amsthm}

\usepackage{todonotes}
\usepackage{cleveref}
\usepackage{commath}

\usepackage{etoolbox}

\makeatletter
\patchcmd{\ttlh@hang}{\parindent\z@}{\parindent\z@\leavevmode}{}{}
\patchcmd{\ttlh@hang}{\noindent}{}{}{}
\makeatother

\theoremstyle{plain}
\newtheorem{theorem}{Theorem}[section]
\newtheorem{lemma}[theorem]{Lemma}
\newtheorem{proposition}[theorem]{Proposition}
\newtheorem{corollary}[theorem]{Corollary}

\theoremstyle{definition}

\newtheorem{example}[theorem]{Example}

\theoremstyle{remark}
\newtheorem{remark}[theorem]{Remark}

\numberwithin{equation}{section}

\def\XXint#1#2#3{{\setbox0=\hbox{$#1{#2#3}{\int}$ }
\vcenter{\hbox{$#2#3$ }}\kern-.6\wd0}}

\newcommand{\T}{\mathbb{T}}
\newcommand{\Z}{\mathbb{Z}}

\DeclareMathOperator{\Span}{span}

\DeclareMathOperator{\SIZ}{SI/Z}
\DeclareMathOperator{\vol}{vol}
\DeclareMathOperator{\Aut}{Aut}
\DeclareMathOperator{\ind}{ind}

\newcommand{\Ppi}{P_{\pi}}
\newcommand{\Prho}{P_{\rho}}

\newcommand{\Hpi}{\mathcal{H}_{\pi}}
\newcommand{\Hr}{\mathcal{H}_{\rho}}

\def\pker#1{P_{#1}}

\title[On sufficient density conditions for lattice orbits of relative discrete series]{On sufficient density conditions for lattice orbits of relative discrete series}

\makeatletter
\@namedef{subjclassname@2020}{\textup{2020} Mathematics Subject Classification}
\makeatother

\subjclass[2020]{22D25, 22E27, 42C30, 42C40, 46L08}

\keywords{Density condition, discrete series, frame, lattice, Riesz sequence}

\author{Ulrik Enstad}
\address{Department of Mathematics,
Stockholm University,
SE-106 91 Stockholm, Sweden.}
\email{ulrik.enstad@math.su.se}

\author{Jordy Timo van Velthoven}
\address{Delft University of Technology,
Mekelweg 4, Building 36,
2628 CD Delft, The Netherlands.}
\email{j.t.vanvelthoven@tudelft.nl}
\begin{document}

\maketitle

\begin{abstract}
This note provides new criteria on a unimodular group $G$ and a discrete series representation $(\pi, \Hpi)$ of formal degree $d_{\pi} > 0$ under which any lattice $\Gamma \leq G$ with $\vol(G/\Gamma) d_{\pi} \leq 1$ (resp. $\vol(G/\Gamma) d_{\pi} \geq 1$) admits $g \in \Hpi$ such that $\pi(\Gamma) g$ is a frame (resp. Riesz sequence). 
The results apply to all projective discrete series of exponential Lie groups. 
\end{abstract}

\section{Introduction}
Let $G$ be a second-countable unimodular group with a lattice $\Gamma \leq G$. For an  irreducible projective unitary representation $(\pi, \Hpi)$ of $G$, let $\pi(\Gamma) g$ be the $\Gamma$-orbit of $g \in \Hpi$, i.e.,
\begin{align*}
\pi(\Gamma) g = \big\{ \pi (\gamma) g : \gamma \in \Gamma \big \}.
\end{align*}
An orbit $\pi(\Gamma) g$ is said to be a \emph{frame} for $\Hpi$ if there exist constants $A, B >0$ such that
\begin{align} \label{eq:frame_ineq}
 A \| f \|_{\Hpi}^2 \leq \sum_{\gamma \in \Gamma} | \langle f, \pi(\gamma) g \rangle |^2 \leq B \|f \|_{\Hpi}^2, \quad f \in \Hpi.
\end{align}
A \emph{Bessel sequence} is a system $\pi(\Gamma) g$ satisfying the upper bound in \eqref{eq:frame_ineq}. The lower bound in \eqref{eq:frame_ineq} implies, in particular, that $g$ is a cyclic vector for the restriction $\pi|_{\Gamma}$. 

A system $\pi(\Gamma) g$ is a \emph{Riesz sequence} if it satisfies inequalities dual to \eqref{eq:frame_ineq}, namely
\begin{align} \label{eq:riesz_ineq}
 A \| c \|_{\ell^2}^2 \leq \bigg\| \sum_{\gamma \in \Gamma} c_{\gamma} \pi(\gamma) g \bigg\|_{\Hpi}^2 \leq B \| c \|_{\ell^2}^2, \quad c \in \ell^2 (\Gamma).
\end{align}
If $\pi(\Gamma) g$ satisfies the upper bound in \eqref{eq:riesz_ineq}, then it is a Bessel sequence. 
The lower bound in \eqref{eq:riesz_ineq} implies, in particular, that a Riesz sequence is linearly independent. 

This note is concerned with the existence of vectors $g \in \Hpi$ such that its orbit $\pi(\Gamma) g$ is a frame or Riesz sequence. A simple necessary condition is that if $\pi(\Gamma) g$ is a frame or Riesz sequence, so that it admits a Bessel constant $B > 0$, then $g \in \Hpi \setminus \{0\}$ satisfies
\[
\int_G | \langle g, \pi(x) g \rangle |^2 \; d\mu_G (x) = \int_{G/\Gamma} \sum_{\gamma \in \Gamma} |\langle \pi(\gamma)^* g, \pi(x) g \rangle |^2 \; d\mu_{G/\Gamma} (x\Gamma) \leq \vol(G/\Gamma) B \| g \|_{\Hpi}^2 < \infty.
\]
An irreducible $\pi$ with a non-zero $L^2$-integrable matrix coefficient is called a (projective) \emph{discrete series}; see Section \ref{sec:discreteseries} for several basic properties. Since nilpotent and (unimodular) exponential Lie groups do not admit genuine representations that are square-integrable in the strict sense, 
the use of projective representations is particularly convenient; see Section \ref{sec:relative_discreteseries}.

In \cite{bekka2004square, enstad2021density}, it has been shown that the existence of frames and Riesz sequences of the form $\pi(\Gamma) g$ can be completely characterized in terms of properties of an associated $\sigma$-twisted convolution operator on $\ell^2 (\Gamma)$, with $\sigma : G \times G \to \mathbb{T}$ being the 2-cocycle of the projective representation $\pi$. An element $\gamma \in \Gamma$ (and its conjugacy class $C_{\Gamma}(\gamma)$ in $\Gamma$) is called \emph{$\sigma$-regular} in $\Gamma$ if $\gamma$ satisfies $\sigma(\gamma, \gamma') = \sigma(\gamma', \gamma)$ whenever $\gamma' \in Z_{\Gamma}(\gamma)$, where $Z_{\Gamma}(\gamma)$ denotes the centralizer of $\gamma$ in $\Gamma$. 

The following theorem contains the main results of \cite{bekka2004square, enstad2021density} for frames and Riesz sequences.

\begin{theorem}[\cite{bekka2004square, enstad2021density}] \label{thm:cdim_unimodular}
Let $(\pi, \Hpi)$ be a discrete series $\sigma$-representation of $G$ of formal degree $d_{\pi} > 0$. Let $\Gamma \leq G$ be a lattice. For a unit vector $\eta \in \Hpi$, define $\phi : \Gamma \to \mathbb{C}$ by
\begin{align*}
\phi(\gamma)= 
\begin{cases}
\dfrac{d_{\pi}}{|C_{\Gamma} (\gamma)|} \displaystyle\int_{G/Z_{\Gamma}(\gamma)} \overline{\sigma(\gamma, y)} \sigma(y, y^{-1} \gamma y) \langle \eta, \pi(y^{-1} \gamma y) \eta \rangle \dif{(yZ_{\Gamma}(\gamma))},
& \begin{matrix} \text{if $C_{\Gamma}(\gamma)$ is finite} \\  \text{and $\sigma$-regular;} \end{matrix} \vspace{0.1cm} \\
0, & \text{otherwise}.
\end{cases}
\end{align*}
Let $C_{\phi}$ be the $\sigma$-twisted convolution operator on $\ell^2 (\Gamma)$ defined by
\begin{align} \label{eq:convolution}
(C_{\phi} c)(\gamma') := \sum_{\gamma \in \Gamma} \phi (\gamma) \sigma(\gamma, \gamma^{-1} \gamma')  c(\gamma^{-1} \gamma'), \quad \gamma' \in \Gamma, \; c \in \ell^2 (\Gamma). 
\end{align}
Then the following assertions hold:
\begin{enumerate}
    \item[(i)] There exists $g \in \Hpi$ such that $\pi(\Gamma) g$ is a frame if and only if $C_{\phi} \leq I_{\ell^2}$; 
    \item[(ii)] There exists $g \in \Hpi$ such that $\pi(\Gamma) g$ is a a Riesz sequence if and only if $C_{\phi} \geq I_{\ell^2}$.
\end{enumerate}
\end{theorem}

The convolution operator $C_{\phi}$ defined in \Cref{thm:cdim_unimodular} determines the so-called \emph{center-valued von Neumann dimension} or \emph{coupling operator} of $\Hpi$ as a module over 
the (twisted) group von Neumann algebra of $\Gamma$. 
Assertions (i) and (ii) in \Cref{thm:cdim_unimodular} are consequences of the 
underlying theory of von Neumann algebras. The paper \cite{bekka2004square}
provides the statements of Theorem \ref{thm:cdim_unimodular} for genuine representations and frames (cf. \cite[Theorem 1]{bekka2004square}), and \cite{enstad2021density} provides an extension to possibly projective representations and Riesz sequences (cf. \cite[Theorem 1.1]{enstad2021density}). 

A direct consequence of Theorem \ref{thm:cdim_unimodular} are the following necessary 
\textquotedblleft density conditions".

\begin{corollary} \label{cor:necessary_density_intro}
With the assumptions and notations as in Theorem \ref{thm:cdim_unimodular},
\begin{enumerate}
    \item[(i)] If there exists $g \in \Hpi$ such that $\pi(\Gamma) g$ is a frame, then $\vol(G/\Gamma) d_{\pi} \leq 1$.
    \item[(ii)] If there exists $g \in \Hpi$ such that $\pi(\Gamma) g$ is a Riesz sequence, then $\vol(G/\Gamma) d_{\pi} \geq 1$. 
\end{enumerate}
\end{corollary}

For a simple proof of Corollary \ref{cor:necessary_density_intro} based on frame and representation theory, see \cite{romero2021density}.

The density conditions provided by \Cref{cor:necessary_density_intro} are generally not sharp, in the sense that they are not sufficient for the existence of frames and Riesz sequences of the form $\pi(\Gamma) g$. For example, this might fail for discrete series of semi-simple Lie groups with a non-trivial center, cf. \cite[Example 1]{bekka2004square}. 
However, for semi-simple Lie groups with a trivial center, the convolution kernel $\phi$ of the operator $C_{\phi}$ in Theorem \ref{thm:cdim_unimodular} is simply given by 
\begin{align} \label{eq:coincidence}
\phi = \vol(G/\Gamma) d_{\pi} \cdot \delta_e .
\end{align}
In general, if the identity \eqref{eq:coincidence} holds, then the density conditions provided by \Cref{cor:necessary_density_intro} are also sufficient for the existence of frames and Riesz sequences. In particular, this holds for lattices in which every non-trivial $\sigma$-regular conjugacy class has infinite cardinality; 
such pairs $(\Gamma, \sigma)$ are sometimes said to satisfy \textquotedblleft Kleppner's condition" \cite{kleppner1962structure}. 

It is the aim of this note to provide new criteria under which the convolution kernel $\phi$ in Theorem \ref{thm:cdim_unimodular} takes the simple form \eqref{eq:coincidence}. In particular, this will imply the optimality of the density conditions provided by \Cref{cor:necessary_density_intro}. 

In order to state the key result of this note, let $B(G)$ be the set of all elements with pre-compact conjugacy classes in $G$. Then $B(G)$ is a normal subgroup of $G$ containing the center $Z(G)$ and was studied for classes of locally compact groups in, e.g., \cite{tits1964automorphismes, greenleaf1974unbounded, sit1975on, moskowitz1978some}. In particular, it was shown that
exponential solvable Lie groups and reductive algebraic groups (with no simple factors) have the property $B(G) = Z(G)$, cf. \Cref{ex:bounded_conjugacy} for references.

The following result provides criteria for a unimodular $G$ with $B(G) = Z(G)$
under which the convolution operator $C_{\phi}$ of \Cref{thm:cdim_unimodular} is a scalar multiple of the identity operator. 

\begin{theorem}\label{thm:cvd_rel}
Let $G$ be such that $B(G) = Z(G)$ and let $\Gamma \leq G$ be a lattice. Suppose that either $G$ is locally connected or $\Gamma$ is co-compact.
Suppose $(\pi, \Hpi)$ is a discrete series of formal degree $d_{\pi} > 0$ such that the projective kernel $\pker{\pi} := \{ x \in G : \pi(x) \in \mathbb{C} \cdot I_{\Hpi} \}$ is trivial. Then the twisted convolution operator $C_{\phi}$ on $\ell^2 (\Gamma)$ defined in \eqref{eq:convolution} is given by
\begin{align} \label{eq:coincidence2}
    C_{\phi} =  \vol(G/\Gamma)d_{\pi} \cdot I_{\ell^2}. 
\end{align} 
Consequently, the following assertions hold:
\begin{enumerate}
    \item[(i)] If $\vol(G/\Gamma) d_{\pi} \leq 1$, then there exists a frame $\pi(\Gamma) g$ for $\Hpi$.  
        \item[(ii)] If $\vol(G/\Gamma) d_{\pi} \geq 1$, then there  exists a Riesz sequence $\pi(\Gamma) g$ in $\Hpi$.  
\end{enumerate}
\end{theorem}

Theorem \ref{thm:cvd_rel} is applicable to all cases in which the necessary density conditions of \Cref{cor:necessary_density_intro} are known to be sharp, namely for abelian groups \cite[Corollary 4.6]{enstad2021density}, linear algebraic semi-simple groups \cite[Theorem 2]{bekka2004square} and square-integrable representations modulo the center of nilpotent Lie groups \cite[Theorem 3]{bekka2004square}. In addition, it is applicable to exponential Lie groups and reductive algebraic groups and and it allows to treat representations that are only square-integrable modulo their projective kernel, since any such a representation is naturally treated as a projective discrete series of the quotient (cf. Section \ref{sec:relative_discreteseries}). 

The assumption in Theorem \ref{thm:cvd_rel} that the projective kernel $\Ppi$ is trivial is essential for its validity. For example, both conclusions (i) and (ii) fail for a holomorphic discrete series $\pi$ of $\mathrm{SL}(2, \mathbb{R})$ (cf. \cite[Example 2]{bekka2004square}), where $\{-I, I \} \subseteq \Ppi$, but \Cref{thm:cvd_rel} is applicable to the (projective) discrete series of $\mathrm{PSL}(2, \mathbb{R}) = \mathrm{SL}(2, \mathbb{R}) / \{ -I, I \}$. On the other hand, for the existence of Riesz sequences $\pi(\Gamma) g$ in general, it is necessary that $\pi|_{\Gamma}$ acts projectively faithful. 

A particular motivation for obtaining \Cref{thm:cvd_rel} was to investigate the optimality of the density conditions in Corollary \ref{cor:necessary_density_intro} for the existence of frames and Riesz sequences for general exponential Lie groups, i.e., Lie groups for which the exponential map is a diffeomorphism. 
For a description of the projective discrete series of an exponential Lie group in terms of the Kirillov correspondence, see \cite{sund1978multiplier, moscovici1978quantization}; in particular, cf. \cite[Proposition 4]{sund1978multiplier}. 

\begin{theorem} \label{thm:main_exponential}
Let $G$ be an exponential solvable Lie group and let $\Gamma \leq G$ be a lattice. Let $(\pi, \Hpi)$ be a projective discrete series of $G$ of formal degree $d_{\pi} > 0$. Then the conclusions of Theorem \ref{thm:cvd_rel} hold.
\end{theorem}

Theorem \ref{thm:main_exponential} covers, in particular, projective representations obtained from genuine representations that are square-integrable modulo the center (see Remark \ref{rem:exponential}).  The existence of frames for such representations of nilpotent Lie groups were shown in \cite{bekka2004square} (cf. \cite[Theorem 3]{bekka2004square} 
and \cite[Corollary 4]{bekka2004square}).
A statement on Riesz sequences does not seem to follow easily from \cite[Theorem 3]{bekka2004square}. 
On the other hand, the existence of Riesz sequences follows transparently from \Cref{thm:main_exponential}, which makes it relevant even for the special case of nilpotent groups. 

A non-nilpotent example to which \Cref{thm:main_exponential} is applicable is given in Section \ref{sec:example}.

\section{Restrictions of discrete series to lattices} \label{sec:restrictions_discreteseries}

Throughout, unless stated otherwise, $G$ denotes a second-countable unimodular locally compact group. A fixed Haar measure on $G$ will be denoted by $\mu_G$. If $H \leq G$ is a closed subgroup with Haar measure $\mu_H$, then there exists a unique $G$-invariant Radon measure $\mu_{G/H}$ on the space $G/H$ of left cosets of $H$ such that Weil's formula holds:
\begin{equation}
    \int_G f(x) \dif{\mu_G(x)} = \int_{G/H} \int_H f(xy) \dif{\mu_H(y)} \dif{\mu_{G/H}(xH)} , \quad f \in L^1(G). \label{eq:weil}
\end{equation}
The measure $\mu_{G/H}$ will always be assumed to be normalized such that \eqref{eq:weil} holds. If $H$ is discrete, then $\mu_H$ will be assumed to be the counting measure.

\subsection{Bounded conjugacy classes}
For a subset $S \subseteq G$, the centralizer of $S$ in $G$ is denoted by $Z_G (S) = \{ x \in G: xs = sx, \; \forall s \in S \}$. In particular, we write $Z_G(x) = Z_G(\{ x \})$ for $x \in G$, and $Z(G) = Z_G (G)$. For $x \in G$, it conjugacy class is $C_G(x) = \{ y x y^{-1} : y \in G \}$. The map $y Z_G(x) \mapsto y x y^{-1}$ is a continuous bijection from $G/Z_G(x)$ onto $C_G(x)$.

The conjugacy class $C_G (x)$ of $x \in G$ is called \emph{bounded} if its closure is compact. The set of all elements $x \in G$ for which $C_G (x)$ is bounded will be denoted by $B(G)$. The set $B(G)$ is a normal subgroup in $G$ containing the center $Z(G)$. 

An automorphism $\alpha \in \Aut(G)$ is said to be of \emph{bounded displacement} if $\{ x^{-1} \alpha (x) : x \in G \}$ is pre-compact. It is readily verified that an inner automorphism $\alpha_y : G \to G, \; x \mapsto y^{-1} x y$ is of bounded displacement if and only if 
$y \in B(G)$. 
If $G$ admits no non-trivial automorphisms of bounded displacement, then $B(G) = Z(G)$. 

Locally compact groups $G$ for which $B(G) = Z(G)$ will play a key role in this note. The following example lists (classes of) groups for which this condition is satisfied.

\begin{example} \label{ex:bounded_conjugacy}
The condition $B(G) = Z(G)$ holds in each of the following cases:
\begin{enumerate}[(a)]
    \item Abelian groups.
    \item Connected, simply connected nilpotent Lie groups (cf. \cite[Theorem 1]{tits1964automorphismes}). 
    \item Exponential solvable Lie groups $G$, i.e., the exponential map $\exp : \mathfrak{g} \to G$ is a diffeomorphism, (cf. \cite[Theorem 9.4]{greenleaf1974unbounded} or \cite[Corollary 1.3]{moskowitz1978some}).
    \item Connected, simply connected complex analytic Lie groups (cf. \cite[Theorem 9.4]{greenleaf1974unbounded}).
    \item Connected semi-simple Lie groups with no compact factors (cf. \cite[Theorem 9.1]{greenleaf1974unbounded}). 
    \item Connected reductive linear algebraic groups with no simple factors (cf. \cite[Theorem 2.4]{sit1975on}).
\end{enumerate}
\end{example}

\subsection{Lattices} 
A discrete subgroup $\Gamma \leq G$ is said to be a \emph{lattice} if the unique invariant Radon measure on $G/\Gamma$ provided by \eqref{eq:weil} is finite. A lattice $\Gamma$ is called \emph{uniform} if $G/\Gamma$ is compact. For classes of amenable groups, including connected solvable Lie groups, 
any lattice is automatically uniform, see \cite{mostow1962homogeneous, bader2019lattices}.

\begin{lemma}\label{lem:lattices}
Let $G$ be such that $B(G) = Z(G)$ and let $\Gamma \leq G$ be a lattice. Suppose that either $G$ is locally connected or that $\Gamma \leq G$ is  uniform. Then the following assertions hold:
\begin{enumerate}[(i)]
    \item For every $\gamma \in \Gamma$, the conjugacy class $C_{\Gamma} (\gamma)$ in $\Gamma$ is either trivial or infinite.
    \item The centralizer $Z_G(\Gamma)$ of $\Gamma$ in $G$ equals the center $Z(G)$ of $G$.
\end{enumerate}
\end{lemma}
\begin{proof}
(i) Let $\gamma \in \Gamma$ be such that $C_{\Gamma} (\gamma )$ is finite. Then $\Gamma/Z_{\Gamma}(\gamma)$ is also finite. 

First, suppose that $G$ is locally connected. Let $\mu_{G/Z_{\Gamma} (\gamma)}$, $\mu_{G/\Gamma}$ and $\mu_{\Gamma / Z_{\Gamma} (\gamma)}$ be the invariant Radon measures on the coset spaces $G/Z_{\Gamma}(\gamma)$, $G/\Gamma$ and $\Gamma/Z_{\Gamma}(\gamma)$, respectively. Since $\mu_{G/\Gamma}$ and $\mu_{\Gamma / Z_{\Gamma} (\gamma)}$ are finite, it follows that also $\mu_{G/Z_{\Gamma} (\gamma)}$ is finite, see, e.g., \cite[Lemma 1.6]{raghunathan1972discrete}.
 Hence, since $Z_{\Gamma}(\gamma) \subseteq Z_G(\gamma)$, the continuous map $G/Z_{\Gamma}(\gamma) \to G/Z_G(\gamma)$ yields a finite $G$-invariant measure on $G/Z_G(\gamma)$. 
An application of \cite[Theorem]{sit1976compactness} yields that $G/Z_G(\gamma)$ is compact. Hence, $C_G(\gamma) \cong G/Z_G(\gamma)$ is compact, so that $\gamma \in B(G) = Z(G)$.

Secondly, if $\Gamma$ is a uniform lattice, then there exists a compact set $\Omega \subseteq G$ such that $G = \Omega \cdot \Gamma$. The conjugacy class $C_G (\gamma)$ is therefore given by
\[
C_G (\gamma) = \{ x \gamma x^{-1} : x \in G \} \subseteq \Omega \cdot C_{\Gamma} (\gamma) \cdot \Omega^{-1}, 
\]
whence pre-compact in $G$. Thus $\gamma \in B(G) = Z(G)$, and hence $C_{\Gamma} (\gamma ) = \{\gamma \}$. 

(ii) Let $x \in Z_G(\Gamma)$, so that $\Gamma \subseteq Z_G(x)$. 
Suppose first that $G$ is locally connected. The finite $G$-invariant measure on $G/\Gamma$ can be pushed forward to a finite $G$-invariant measure on $G/Z_G(x)$, which implies that $G/Z_G(x)$ is compact by \cite[Theorem]{sit1976compactness}. Hence, $C_G(x) \cong G/Z_G(x)$ is compact, and thus $x \in B(G) = Z(G)$.

Lastly, if $\Gamma$ is a uniform lattice, then the continuous surjective map $G / \Gamma \to G/Z_G(x)$ yields that $G/Z_G(x)$ is compact. Therefore, the continuous bijection $G/Z_G(x) \to C_G(x)$ yields that also $C_G (x)$ is compact,
and thus $x \in B(G) = Z(G)$.
\end{proof}

Lemma \ref{lem:lattices} applies, in particular, to arbitrary lattices in Lie groups. For this setting, there are alternative proofs of the used \cite[Theorem]{sit1976compactness}, see \cite[Theorem 1]{greenleaf1975compactness} and \cite[Theorem 2]{greenleaf1974automorphisms}. It is not known whether \Cref{lem:lattices} 
holds for non-uniform lattices in general unimodular groups.

\subsection{Projective discrete series} \label{sec:discreteseries}
A \emph{projective unitary representation} $(\pi, \Hpi)$ of $G$ on a separable Hilbert space $\Hpi$ is a strongly measurable map $\pi : G \to \mathcal{U}(\Hpi)$ satisfying
\[ \pi(x)\pi(y) = \sigma(x,y)\pi(xy), \quad x,y \in G, \]
for some function $\sigma : G \times G \to \mathbb{T}$. The function $\sigma$ necessarily forms a \emph{$2$-cocycle} on $G$, that is, it is a Borel function satisfying the identities
\begin{align*}
 \sigma(e, e) = 1 \quad \text{and} \quad   \sigma(x,y)\sigma(xy,z) &= \sigma(x,yz)\sigma(y,z) \quad \text{for all $x,y,z \in G$}. 
\end{align*}
A projective unitary representation with 2-cocycle $\sigma$ will simply be referred to as a $\sigma$-representation. For $\sigma \equiv 1$, it will simply be said that $\pi$ is a \emph{representation}.

A $\sigma$-representation $(\pi, \Hpi)$ is \emph{irreducible} if the only closed $\pi(G)$-invariant subspaces are $\{0\}$ and $\Hpi$. It is called \emph{square-integrable} if there exist nonzero $f,g \in \Hpi$ such that
\[ \int_G | \langle f, \pi(x) g \rangle |^2 \; d \mu_G (x) < \infty . \]
An irreducible, square-integrable $\sigma$-representation is called a \emph{discrete series $\sigma$-representation}, or a \emph{projective discrete series} if the associated cocycle is irrelevant.

The significance of a discrete series $\pi$ is the existence of a unique $d_{\pi} > 0$, called its \emph{formal degree}, such that the orthogonality relations
\[ \int_G \langle f, \pi(x) g \rangle \overline{ \langle f', \pi(x) g' \rangle } \; d \mu_G (x) = d_{\pi}^{-1} \langle f, f' \rangle \overline{ \langle g, g' \rangle } \]
hold for all $f,f',g,g' \in \Hpi$. 

\subsection{The projective kernel}
The \emph{projective kernel} of a $\sigma$-representation $\pi$ is defined by
\begin{equation}
    \pker{\pi} := \{ x \in G : \pi(x) \in \T \cdot I_{\Hpi} \}.  \label{eq:proj_ker}
\end{equation}
The $\sigma$-representation $\pi$ is \emph{projectively faithful} if $\pker{\pi} = \{e\}$.
 
Throughout, $\chi_\pi : \pker{\pi} \to \T$ denotes the measurable function satisfying $\pi(x) = \chi_\pi(x)I_{\Hpi}$ for all $x \in \pker{\pi}$. Then, for $f,g \in \Hpi$, $x \in G$ and $y \in \pker{\pi}$, 
\begin{equation}
    |\langle f, \pi(xy) g \rangle| = |\langle f, \overline{\sigma(x,y)} \pi(x) \chi_\pi(y) g \rangle|  = | \langle f, \pi(x) g \rangle |, \label{eq:invariance}
\end{equation}
so that $x\Ppi \mapsto |\langle f, \pi (x) g \rangle |$ is a well-defined function on the coset space $G/\Ppi$. 
 
\begin{lemma} \label{lem:proj_ker0}
 If $\pi$ is a $\sigma$-representation of $G$, then $\Ppi$ is a closed normal subgroup. If, in addition, $\pi$ is square-integrable, then $\Ppi$ is compact. 
 \end{lemma}
 \begin{proof}
 Let $\mathcal{P}(\Hpi) := \mathcal{U}(\Hpi) / \mathbb{T} \cdot I_{\Hpi}$ be the projective unitary group of $\Hpi$, equipped with the quotient topology relative to the strong operator topology on $\mathcal{U}(\Hpi)$. Let $p : \mathcal{U}(\Hpi) \to \mathcal{P}(\Hpi)$ be the canonical projection. 
 By \cite[Theorem 7.5]{varadarjan1985geometry}, the map $\pi' := p \circ \pi : G \to \mathcal{P}(\Hpi)$ is a continuous homomorphism, and hence $\pker{\pi} = \ker(\pi')$ is a closed normal subgroup. 
 
Suppose $\pi$ is square-integrable. Letting $f,g \in \Hpi \setminus \{0\}$, we apply \eqref{eq:weil} and \eqref{eq:invariance} to obtain
\begin{align*}
  \infty >  \int_G | \langle f, \pi(x) g \rangle|^2 \dif{\mu_G(x)} = \int_{\pker{\pi}} \dif{\mu_{\pker{\pi}}(y)} \int_{G/\pker{\pi}} |\langle f, \pi(x) g \rangle |^2 \dif{\mu_{G/\pker{\pi}}(x\pker{\pi})},
\end{align*}
and thus the Haar measure of $\pker{\pi}$ is finite, so that $\pker{\pi}$ must be compact. 
\end{proof}

Following \cite{kleppner1962structure}, $x \in G$ is called \emph{$\sigma$-regular} in $G$ if $\sigma(x,y) = \sigma(y,x)$ for all $y \in Z_G (x)$.

\begin{lemma}\label{lem:proj_ker}
Let $G$ be such that $B(G) = Z(G)$ and let $(\pi, \Hpi)$ be a discrete series $\sigma$-representation of $G$. 
Then the projective kernel coincides with the $\sigma$-regular elements of the center of $G$. In particular, the following are equivalent:
\begin{enumerate}
    \item[(i)] $\pi$ is projectively faithful.
    \item[(ii)] The only $\sigma$-regular element of $G$ with precompact conjugacy class is the identity.
\end{enumerate}
\end{lemma}

\begin{proof}
If an element $x \in Z(G)$ is $\sigma$-regular, then
\[ \pi(x)\pi(y) = \sigma(x,y) \pi(xy) = \sigma(y,x) \pi(yx) = \pi(y)\pi(x) \]
for all $y \in G$. Thus $\pi(x) \in \pi(G)' = \T \cdot I_{\Hpi}$ by irreducibility, so that $x \in \pker{\pi}$.
Conversely, if $x \in \pker{\pi}$, then $C_G(x) \subseteq \pker{\pi}$ since $\pker{\pi}$ is a normal subgroup of $G$. Since $\pker{\pi}$ is compact by \Cref{lem:proj_ker0}, it follows that $C_G(x)$ is pre-compact, hence $x \in B(G) = Z(G)$. Therefore, for any $y \in G$,
\begin{align*}
     \chi_\pi(x)\pi(y) &= \pi(x)\pi(y) = \sigma(x,y)\pi(xy) = \sigma(x,y) \pi(yx) \\
     &= \sigma(x,y)\overline{\sigma(y,x)} \pi(y)\pi(x) = \sigma(x,y)\overline{\sigma(y,x)} \chi_\pi(x) \pi(y) .
\end{align*}
Hence $\sigma(x,y) = \sigma(y,x)$ for all $y \in G$, so $x$ is $\sigma$-regular.

Since $B(G) = Z(G)$, it follows that the projective kernel coincides with $\sigma$-regular elements with pre-compact conjugacy classes. In particular, $\pker{\pi} = \{ e \}$ if and only if the only $\sigma$-regular element with pre-compact conjugacy class is the identity.
\end{proof}

A combination of the previous lemmata allows a proof of \Cref{thm:cvd_rel}: 

\begin{proof}[Proof of \Cref{thm:cvd_rel}]
Suppose $\gamma \in \Gamma$ is $\sigma$-regular in $\Gamma$ and $C_{\Gamma} (\gamma)$ is finite. For showing \eqref{eq:coincidence2}, we have to show that $\phi(\gamma) = \vol(G/\Gamma)d_{\pi} \delta_{\gamma,e}$. By \Cref{lem:lattices}(i), it follows that $C_{\Gamma} (\gamma)$ is trivial, so that $\gamma$ is in the center $Z(\Gamma)$ of $\Gamma$. In particular, this implies that $Z_{\Gamma} (\gamma) = \Gamma$. 
In addition, \Cref{lem:lattices}(ii) yields that $\gamma \in Z(G)$. Hence,
\begin{align*} \phi(\gamma) &= d_{\pi} \int_{G/\Gamma} \overline{\sigma(\gamma, y)} \sigma(y, y^{-1} \gamma y) \langle \eta, \pi(y^{-1} \gamma y) \eta \rangle \; d\mu_{G/\Gamma} {(y\Gamma)} \\
&= d_{\pi} \langle \eta, \pi(\gamma) \eta \rangle \int_{G/\Gamma} \overline{\sigma(\gamma,y)} \sigma(y,\gamma) \; d \mu_{G/\Gamma} {(y\Gamma)}.
\end{align*}
The function $\omega_{\gamma} \colon G \to \T$ given by $\omega_{\gamma} (y) = \overline{\sigma(\gamma,y)} \sigma(y,\gamma)$ is a homomorphism, since
\begin{align*}
    \omega_{\gamma}(yy') &= \overline{\sigma(\gamma,yy')} \sigma(yy',\gamma) = \overline{\sigma(\gamma,yy') \sigma(y,y')} \sigma(y,y') \sigma(yy',\gamma) \\
    &= \overline{\sigma(\gamma,y)\sigma(\gamma y,y')}  \sigma(y,y'\gamma) \sigma(y',\gamma) = \overline{ \sigma(\gamma,y)\sigma(y \gamma,y') \sigma(\gamma,y')} \sigma(y,\gamma y') \sigma(\gamma,y') \sigma(y',\gamma) \\
    &= \overline{\sigma(\gamma,y)\sigma(y\gamma,y') \sigma(\gamma,y')}  \sigma(y,\gamma)\sigma(y\gamma,y') \sigma(y',\gamma) = \overline{\sigma(\gamma,y) }  \sigma(y,\gamma) \overline{ \sigma(\gamma,y')} \sigma(y',\gamma) \\
    &= \omega_{\gamma} (y)\omega_{\gamma} (y').
\end{align*}
The $G$-invariance of the measure on $G/\Gamma$ gives that
\[ \int_{G/\Gamma} \omega_{\gamma} (y) \dif{\mu_{G/\Gamma}(y\Gamma)} = \int_{G/\Gamma} \omega_{\gamma} (y'y) \dif{\mu_{G/\Gamma}(y\Gamma)} = \omega_{\gamma} (y') \int_{G/\Gamma} \omega_{\gamma} (y) \dif{\mu_{G/\Gamma}(y\Gamma)} , \quad y' \in G, \]
which means that \[ \int_{G/\Gamma} \omega_{\gamma} (y\Gamma) \; d{\mu_{G/\Gamma}(y\Gamma)} = \begin{cases} \vol(G/\Gamma), & \text{if $\omega_{\gamma} \equiv 1$,} \\ 0, & \text{otherwise.} \end{cases} \]
Note that $\omega_{\gamma} \equiv 1$ if and only if $\sigma(\gamma,y) = \sigma(y,\gamma)$ for all $y \in G$, i.e., if and only if $\gamma$ is $\sigma$-regular in $G$. Since $\gamma \in Z(G)$, $\gamma$ is $\sigma$-regular in $G$ if and only if $\gamma = e$ by \Cref{lem:proj_ker} and the assumption that $\pi$ is projectively faithful. Hence,
\[ \phi(\gamma) = \vol(G/\Gamma) d_{\pi} \cdot \delta_{\gamma,e} , \]
as required.
\end{proof}

\begin{proof}[Proof of \Cref{thm:main_exponential}]
Since $G$ is exponential, we have $Z(G) = B(G)$ by \cite[Theorem 9.4]{greenleaf1974unbounded}. By \Cref{lem:proj_ker0}, the projective kernel $\Ppi$ of a discrete series $\pi$ must be compact, hence trivial, since $G$ does not contain nontrivial compact subgroups, see, e.g., \cite[Theorem 14.3.12]{hilgert2012structure}.
The conclusion follows therefore directly from \Cref{thm:cvd_rel}.
\end{proof}

\section{Discrete series modulo the projective kernel} \label{sec:relative_discreteseries}
This section considers projective representations obtained from genuine representations that are square-integrable modulo their projective kernel. Such projective representations are projectively faithful, 
and they form an important class to which \Cref{thm:cvd_rel} applies.

Let $(\rho,\mathcal{H}_{\rho})$ be an irreducible representation of a second countable unimodular group $H$.
It is called a \emph{relative discrete series} (modulo $\Prho$) if there exist non-zero $f, g \in \Hr$ such that
\[ \int_{H/\Prho}|\langle f, \rho(\dot{x}) g \rangle|^2 \; d\mu_{H/\Prho} (\dot{x}) < \infty ,\]
where $\dot{x} = x\Prho$ and $\mu_{H/\Prho}$ denotes Haar measure on $H/\Prho$. 

A relatively discrete series $(\rho, \Hr)$ of $H$ can be treated as a (projective) discrete series of $G := H / \Prho$. For this, choose a Borel section $s \colon  H/\Prho \to H$ of the canonical quotient map, and set $\pi := \rho \circ s$. Then a direct calculation shows that
\[ \pi(\dot{x})\pi(\dot{y}) = \sigma(\dot{x},\dot{y})\pi(\dot{x} \dot{y}) , \quad \dot{x}, \dot{y} \in G = H/\Prho, \]
where the 2-cocycle $\sigma$ is given by
\[ \sigma(\dot{x},\dot{y}) = \chi_{\rho} ( s(\dot{x})s(\dot{y})s(\dot{x} \dot{y})^{-1}), \quad \dot{x}, \dot{y} \in G = H/\Prho. \]
A different choice of section $s$ yields a 2-cocycle cohomologous to $\sigma$ and a representation unitarily equivalent to $\pi$. 

The following proposition is a special case of \Cref{thm:cvd_rel}.

\begin{proposition}\label{thm:cvd_rel2}
Let $(\rho, \Hr)$ be a relative discrete series (modulo $\Prho$) of $H$ and let $(\pi, \Hpi)$ be the associated $\sigma$-representation of $G = H/\Prho$. Let $\Gamma \leq G$ be a lattice. Suppose that 
$B(G) = Z(G)$ and that either $G$ is locally connected or that $\Gamma$ is uniform. Then 
$ C_{\phi} =  \vol(G/\Gamma)d_{\pi} \cdot I_{\ell^2} .$
\end{proposition}
\begin{proof}
Denote by $\pi$ a $\sigma$-representation of $G = H /\Prho$ obtained from $\rho$ via a Borel section $s$. If $x\Prho \in \Ppi \leq H/\Prho$, then $\rho(s(x\Prho)) = \pi(x\Prho) \in \T \cdot I_{\Hpi}$, so that $s(x\Prho) \in \Prho$. This implies that $x\Prho = s(x\Prho) \Prho = \Prho$, and thus 
$\Ppi = \{ e \Prho \}$. Therefore, it follows from \Cref{thm:cvd_rel} that $C_{\phi} = \vol(G/\Gamma) d_{\pi} \cdot I_{\ell^2}$.
\end{proof}

\begin{remark} \label{rem:exponential}
The projective kernel of an irreducible representation 
$(\rho, \Hr)$ of an exponential Lie group $H$ is connected 
by \cite[Theorem 2.1]{bekka1990complemented}, and thus $G= H/\Prho$ is again exponential and $B(G) = Z(G)$ by \cite[Theorem 9.4]{greenleaf1974unbounded}. 
In the particular case that $\rho$ is square-integrable modulo the center
$Z(H)$, then $\pker{\rho} = Z(H)$ by \cite[Theorem 2.1]{bekka1990complemented} combined with \cite[Theorem 5.3.4]{duflo1976sur} or \cite[Section 4.1]{rosenberg1978square}.
\end{remark}

\section{A non-nilpotent exponential $\SIZ$ group with a lattice.} \label{sec:example}

This section provides an example of a non-nilpotent exponential Lie group to which Theorem \ref{thm:main_exponential} is applicable, showing that it is not vacuous in the non-nilpotent case. 

\subsection{Completely solvable group}
Let $\mathfrak{g} = \Span_{\mathbb{R}} \{X_1, ..., X_5\}$ with non-zero Lie brackets
\[
[X_2, X_3]=X_1, \; [X_2, X_5] = X_2, \; [X_3, X_5 ] = - X_3, \; [X_4, X_5] = X_1.
\]
Then $\mathfrak{g}$ is completely solvable, i.e., it admits a sequence of ideals
\[
\{ 0 \} = \mathfrak{g}_0 \subset \mathfrak{g}_1 \subset \cdots \subset \mathfrak{g}_{4} \subset \mathfrak{g}_5 = \mathfrak{g}, \quad \text{with} \quad \dim(\mathfrak{g}_j) = j. 
\]
In particular, this shows that $\mathfrak{g}$ is an exponential solvable Lie algebra. Its  nilradical is given by $\mathfrak{n} = \Span_{\mathbb{R}} \{X_1, X_2, X_3 \} \oplus \mathbb{R} X_4$, so that $\mathfrak{g}$ is non-nilpotent. The center of $\mathfrak{g}$ is $\mathfrak{z} (\mathfrak{g}) = \mathbb{R} X_1$. 

Let $G$ be the connected, simply connected Lie group with Lie algebra $\mathfrak{g}$. Denote by $N$ and $T$ the connected Lie subgroups with Lie algebras $\mathfrak{n}$ and $\mathbb{R} X_5$, respectively. Then $G$ is a semi-direct product $G = N T$ with group multiplication 
\begin{align*}
    \begin{pmatrix}
    x \\ y \\ z \\ w \\ t
    \end{pmatrix} 
    \cdot 
    \begin{pmatrix}
    x' \\ y' \\ z' \\ w' \\ t'
    \end{pmatrix}
    = 
    \begin{pmatrix}
    x +x' - t w' - \frac{1}{2}(e^t z y' - e^{-t} y z') \\
    y + e^{-t} y' \\
    z + e^t z' \\
    w + w' \\
    t + t'
    \end{pmatrix}
\end{align*}
The center of $G$ is given by $Z(G) = \{ (x,0,0,0,0) : x \in \mathbb{R} \}$. 


\subsection{Lattice}

A lattice in $G$ can be given as follows, cf.\ \cite[p.\ 237]{diatta15lattices}: Let $n \in \mathbb{N}$, $n > 2$ and let $x_0 > 0$ be such that $(e^{t_0})^2 - ne^{t_0} + 1 = 0$. Let $v_1 = (1,1)$ and $v_2 = (e^{-t_0},e^{t_0})$ and $x_0 = (e^{-t_0} - e^{t_0})/2$. Then the set
\[ \Gamma = \Z \begin{pmatrix} x_0 \\ 0 \\ 0 \\ 0 \\ 0 \end{pmatrix} + \Z \begin{pmatrix} 0 \\ v_1 \\ 0 \\ 0 \end{pmatrix} + \Z \begin{pmatrix} 0 \\ v_2 \\ 0 \\ 0 \end{pmatrix} + \Z \begin{pmatrix} 0 \\ 0 \\ 0 \\ x_0/t_0 \\ 0 \end{pmatrix} + \Z \begin{pmatrix} 0 \\ 0 \\ 0 \\ 0 \\ t_0 \end{pmatrix} \]
forms a lattice in $G$.

The projection of this lattice via the quotient map to $G/Z$ yields a lattice in $G/Z$.

\subsection{Relative discrete series}

For showing that $G$ admits relative discrete series representations, it suffices (cf. \cite[\S 4.1]{rosenberg1978square}) to show that there exists a functional $\ell \in \mathfrak{g}^*$ such that 
\[
\mathfrak{g} (\ell) := \{X \in \mathfrak{g} : \ell([Y,X]) = 0, \; \forall \; Y \in \mathfrak{g} \} = \mathfrak{z}(\mathfrak{g}). 
\]
Denoting by $(\xi_1, ..., \xi_5)$ a dual basis of $\{X_1, ..., X_5\}$ in $\mathfrak{g}^*$, it is readily verified that $\mathfrak{g}(\xi_1) = \mathfrak{z}(\mathfrak{g})$. 

For $\xi := \xi_1 \in \mathfrak{g}^*$, consider the associated polarization $\mathfrak{p} = \Span_{\mathbb{R}} \{X_1, X_2, X_4 \}$. Then
\[
 \chi_{\xi} (\exp(X)) = e^{2\pi i \xi(X)}, \quad X \in \mathfrak{p},
\]
defines a unitary character of $P := \exp (\mathfrak{p}) \leq G$. The Kirillov-Bernat correspondence yields that  the induced representation $\rho_{\xi} = \ind_P^G (\chi_{\xi})$ forms an irreducible representation of $G$ that is square-integrable modulo $Z(G)$.




\section*{Acknowledgements}
U.E.\ gratefully acknowledges support from the The Research Council of Norway through project 314048. J.v.V. gratefully acknowledges support from 
the Austrian Science Fund (FWF) project J-4445.

\bibliographystyle{abbrv}
\bibliography{bib}

\end{document}